\documentclass[12pt,a4]{amsart}
\usepackage{amsmath}
\usepackage{amssymb}
\usepackage{amsthm}
\usepackage{enumitem}
\usepackage{url}
\usepackage{soul}

\usepackage{color}
\usepackage{dsfont}
\usepackage{epsfig}
\usepackage[hidelinks]{hyperref}
\usepackage{setspace}
\usepackage[authoryear,round,sort]{natbib}
\usepackage{comment}
\usepackage{booktabs}
\usepackage{float}
\usepackage{stmaryrd}
\numberwithin{equation}{section}

\usepackage{geometry}
\theoremstyle{plain}
\newtheorem{proposition}{Proposition}[section]

\newtheorem{theorem}{Theorem}[section]
\newtheorem{lemma}{Lemma}[section]
\newtheorem{corollary}{Corollary}[section]

\theoremstyle{definition}
\newtheorem{definition}{Definition}[section]

\theoremstyle{remark}

\expandafter\let\expandafter\oldproof\csname\string\proof\endcsname
\let\oldendproof\endproof
\renewenvironment{proof}[1][\proofname]{%
  \oldproof[\noindent\textbf{#1.} ]%
}{\oldendproof}
\newcommand{\1}{\mathds{1}}

\newcommand{\E}{\mathbb{E}}

\newcommand{\be}{\begin{equation}}
\newcommand{\ee}{\end{equation}}
\newcommand{\by}{\begin{eqnarray*}}
\newcommand{\ey}{\end{eqnarray*}}

\renewcommand{\leq}{\leqslant}
\renewcommand{\geq}{\geqslant}
\usepackage{xcolor}
\definecolor{dark-red}{rgb}{0.4,0.15,0.15}
\definecolor{dark-blue}{rgb}{0.15,0.15,0.4}
\definecolor{medium-blue}{rgb}{0,0,0.5}
\hypersetup{
    colorlinks, linkcolor={red},
    citecolor={blue}, urlcolor={blue}
}
\allowdisplaybreaks

\begin{document}
	\title{Hitting time, access time and optimal transport on graphs}
	\author{Michael C.H. Choi}\thanks{}
	\address{Institute for Data and Decision Analytics, The Chinese University of Hong Kong, Shenzhen, Guangdong, 518172, P.R. China}
	\email{michaelchoi@cuhk.edu.cn}
	\date{\today}
	\maketitle
	
	\begin{abstract}
		Given a discrete source distribution $\mu$ and discrete target distribution $\nu$ on a common finite state space $\mathcal{X}$, we are tasked with transporting $\mu$ to $\nu$ using a given discrete-time Markov chain $X$ with the quickest possible time on average. We define the optimal transport time $H(\mu,\nu)$ as stopping rule of $X$ that gives the minimial expected transport time. This is also known as the access time from $\mu$ to $\nu$ of $X$ in [L. Lov\'{a}sz and P. Winkler. Efficient Stopping Rules for Markov Chains. Proceedings of the Twenty-seventh Annual ACM Symposium on Theory of Computing (STOC '95) 76-82.]. We study bounds of $H(\mu,\nu)$ in various special graphs, which are expressed in terms of the mean hitting times of $X$ as well as parameters of $\mu$ and $\nu$ such as their moments. Among the Markov chains that we study, random walks on complete graphs is a good choice for transport as $H(\mu,\nu)$ grows linearly in $n$, the size of the state space, while that of the winning streak Markov chain exhibits exponential dependence in $n$.
		\smallskip
		
		
	\end{abstract}
	
	
	
	\section{Introduction and main results}
	
	Suppose that we are given a discrete-time ergodic Markov chain $X = (X_n)_{n \in \mathbb{N}}$ on a finite state space $\mathcal{X}$ with transition matrix $P = (p_{i,j})_{i,j \in \mathcal{X}}$ that converges to the stationary distribution $\pi = (\pi_i)_{i \in \mathcal{X}}$. We are tasked with transporting from a source distribution $X_0 \sim \mu = (\mu_i)_{i \in \mathcal{X}}$ to a given target distribution $X_T \sim \nu = (\nu_i)_{i \in \mathcal{X}}$ using the shortest possible time on average, where $T$ is a stopping time of $X$. Note that $\mu$ and $\nu$ need not have a common support as long as $\sum_{i \in \mathcal{X}} \mu_i = \sum_{i \in \mathcal{X}} \nu_i = 1.$
	Writing $\Gamma$ to be the set of stopping time of $X$ that transports $\mu$ to $\nu$, we say that the stopping time $T$ is an optimal transport time if it minimizes the mean transport time among $\Gamma$, that is, 
	\begin{definition}[optimal transport time]\label{def:Hmunu}
		$$H(\mu,\nu) := \inf_{T \in \Gamma} \E[T].$$
	\end{definition}
	In the literature, $H(\mu,\nu)$ is known as the access time from $\mu$ to $\nu$ of $X$ \cite{LW95}.  Quite a few well-studied parameters of Markov chains can be formulated in this way. For example, when $\nu = \pi$, this is known as the ``mixing" time of $X$ in \cite{LW98} (This is not to be confused with the classical total-variation mixing time of $X$). On the other hand, if $\mu = \delta_i$ and $\nu = \delta_j$, the Dirac mass at state $i$ and $j$ respectively, then 
	$$H(i,j) := H(\delta_i,\delta_j) = \E_i[\tau_j]$$
	is the mean hitting time of $X$ from $i$ to $j$, where we denote $\tau_j := \inf\{n \geq 0;~X_n = j\}$ and the usual convention of $\inf \emptyset = 0$ applies. It is also obvious to see that, for $j \in \mathcal{X}$,
	$$H(\mu,j) := H(\mu,\delta_j) = \sum_{i \in \mathcal{X}} \mu_i H(i,j).$$
	A closely related problem is the Skorohod embedding problem, where the underlying process is a Brownian motion instead of a finite Markov chain that we consider. We refer interested readers to \cite{O04} for an excellent survey in this direction, and to \cite{MR16} in which they study the Skorohod embedding problem with non-randomized stopping time $T$ in a Markov chain setting.
	A simple yet useful characterization of $H(\mu,\nu)$ that this paper frequently exploits is the following representation:
	\begin{theorem}[Theorem $5.2$ of \cite{LW95}]\label{thm:Hmunu}
		$$H(\mu,\nu) = \max_{j \in \mathcal{X}} \sum_{i \in \mathcal{X}} (\mu_i - \nu_i) \E_i[\tau_j].$$
	\end{theorem}
	In view of Theorem \ref{thm:Hmunu}, $H(\mu,\nu)$ can be interpreted as a generalized notion of mean hitting time. Note that $H(\mu,\nu)$ is not a metric on the space of probability measures, yet it defines a statistical divergence: it is clear from Definition \ref{def:Hmunu} that $H(\mu,\nu) \geq 0$ and is $0$ if and only if $\mu = \nu$. It is also trivial from Theorem \ref{thm:Hmunu} that the triangle inequality $H(\mu,\nu) \leq H(\mu,\rho) + H(\rho,\nu)$ holds where $\rho$ is also a distribution on $\mathcal{X}$. However, it is in general asymmetric, that is, $H(\mu,\nu) \neq H(\nu,\mu)$.
	
	The aim of this paper is to investigate the extend to which $H(\mu,\nu)$ depends on the structure of the state space as well as to the initial and target distribution $\mu$ and $\nu$. Writing $n := |\mathcal{X}|$ to be the size of the state space, one important aspect of our investigation is how $H(\mu,\nu)$ depends on $n$. This can shed lights on the time complexity of such transportation. In the following, we summarize our results in a table format:
	\begin{table}[htbp]
		\centering
		\caption{Main results}
		\begin{tabular}{l|c|c}
		$X$	&  $H(\mu,\nu)$     & Result \\
		\hline
		symmetric birth-death processes with parameter $p$	&  $O\left(\frac{n^2}{p}\right)$     & Theorem \ref{thm:bd}  \\
		winning streak Markov chain	& $O(2^n)$  & Theorem \ref{thm:ws} \\
		random walk on $d$-dimensional hypercube & $O(n)$      & Corollary \ref{cor:hypercube}  \\
		random walk on connected graphs	& $O(n^3)$  & Corollary \ref{cor:rwgraph} \\
		random walk on the $n$-path	& $O(n^2)$  & Theorem \ref{thm:npath} \\
		random walk on complete graphs	& $O(n)$  & Theorem \ref{thm:cgraph} \\
		random walk on the $n$-star	& $O(n)$ &  Theorem \ref{thm:nstar}
		\end{tabular}%
		\label{tab:main}%
	\end{table}%
	
	Among various Markov chains that we study in this paper, Table \ref{tab:main} seems to suggest that random walk on complete graphs and on the $n$-star are more suitable for such transport as the complexity of $H(\mu,\nu)$ is at most $O(n)$. This is reasonable as walks on these graph structures
	can readily access most parts of the state space relatively quickly. This stands in contrast with walks on linear graphs such as $n$-path or symmetric birth-death processes, in which they can have worst-case optimal transport time to be of the order $O(n^2)$ as the transitions are restricted to nearest-neighbour type. The worst example that we study is the winning streak Markov chains. It exhibits exponential dependence on $n$ in the worst case.
	
	The rest of the paper is organized as follows. In Section \ref{subsec:notations}, we fix commonly used notations throughout the paper. In Section \ref{sec:opt}, we present our results on $H(\mu,\nu)$ under different structures on the state space, namely for general finite Markov chains in Section \ref{subsec:mc}, symmetric birth-death processes in Section \ref{subsec:bd}, the winning streak Markov chain in Section \ref{subsec:ws}, random walk on $d$-dimensional hypercube in Section \ref{subsec:hc} and random walks on connected graphs in Section \ref{subsec:graphs}. Afterwards, we proceed our investigation to various special graphs such as random walk on the $n$-path (Section \ref{subsec:npath}), complete graph (Section \ref{subsec:cgraph}) and $n$-star (Section \ref{subsec:nstar}).
	
	\subsection{Notations}\label{subsec:notations}
	Throughout the paper, we write $\1_A$ to be the indicator function of the set $A$, and for $x \in \mathbb{R}$ we denote $x_+ := \max\{x,0\}$. For $a,b \in \mathbb{Z}$, we write $\llbracket a,b \rrbracket := \{a,a+1,\ldots,b-1,b\}$. For any two probability measures $\mu, \nu$ on $\mathcal{X}$, we define the total variation distance between $\mu$ and $\nu$ to be
	$$||\mu - \nu||_{TV} := \sup_{A \subset \mathcal{X}} |\mu(A) - \nu(A)| = \dfrac{1}{2} \sum_{j \in \mathcal{X}} |\mu(j) - \nu(j)|.$$
	We also denote $Z \sim \mu$ meaning the law of the random variable is $\mu$, and denote the expectation under $\mu$ to be $\E_{Z \sim \mu}$.
	
	\section{Optimal transport on ...}\label{sec:opt}
	
	\subsection{... finite Markov chains}\label{subsec:mc}
	
	In this section, we consider a general finite Markov chain $X$ and the task is to transport $\mu$ to $\nu$ using $X$. Our main result in this section below essentially states that we can bound the optimal transport time $H(\mu,\nu)$ by the maximal mean hitting time. This result is particularly useful as for many Markov chains or random walks on graphs, information for maximal mean hitting time is readily available. We will see how we can apply this bound in subsequent sections which may give tight estimates on the order of $H(\mu,\nu)$.
	
	\begin{theorem}\label{thm:mcmain}
		Suppose that $X$ is a finite Markov chain on $\mathcal{X}$. For the optimal transport time $H(\mu,\nu)$ from $\mu$ to $\nu$, we have
		\begin{align}\label{eq:mcmain}
		H(\mu,\nu) \leq \max_{i,j \in \mathcal{X}} \mathbb{E}_i[\tau_j].
		\end{align}
	\end{theorem}
	
	\begin{proof}
		For a given state $j \in \mathcal{X}$, we see that
		$$H(\mu,j) - H(\nu,j) \leq H(\mu,j) = \sum_{i \in \mathcal{X}} \mu_i H(i,j) \leq \max_{i,j \in \mathcal{X}} \mathbb{E}_i[\tau_j].$$
		Maximizing over $j$ together with Theorem \ref{thm:Hmunu} yields the desired result.
	\end{proof}
	
	Interestingly, this upper bound is independent of $\mu$ or $\nu$, and depends on the choice of Markov chain through the maximal mean hitting time.
	
	\subsection{... symmetric birth-death processes}\label{subsec:bd}
	
	In this section, we present results when the underlying Markov chain $X$ is taken to be a symmetric birth-death process on $\mathcal{X} = \llbracket 0,n \rrbracket$. More precisely, the transition dynamics of $X$ is assumed to be
	\begin{align}\label{eq:bd}
	p_{i,j} = \begin{cases}
	\textrm{(birth probability) } p, \quad \textrm{if } j = i + 1, i \in \llbracket0,n-1 \rrbracket,\\
	\textrm{(death probability) } q := p, \quad \textrm{if } j = i - 1, i \in \llbracket1,n \rrbracket,\\
	\textrm{(holding probability) }r := 1-2p, \quad \textrm{if } j = i, i \in \llbracket0,n \rrbracket,
	\end{cases}
	\end{align}
	where $p \in (0,1/2]$ and $p_{i,j} = 0$ otherwise. In other words, $X$ has the same birth probability and death probability at each state, and this probability is parametrized by $p$.
	\begin{theorem}\label{thm:bd}
		Suppose that $X$ is a birth-death process with the above transition dynamics \eqref{eq:bd}. For the optimal transport time $H(\mu,\nu)$ from $\mu$ to $\nu$ using $X$, we have
		\begin{align}\label{eq:bdmain}
		H(\mu,\nu) = \dfrac{1}{2p} \bigg(&\mathbb{E}_{Z \sim \nu}[Z] - \mathbb{E}_{Z \sim \mu}[Z] + \max_{j \in \llbracket 0,n \rrbracket} \bigg(\mathbb{E}_{Z \sim \mu}\left[\max\{Z,j\}^2 - \min\{Z,j\}^2 \right] \\
		&\quad - \mathbb{E}_{Z \sim \nu}\left[\max\{Z,j\}^2 - \min\{Z,j\}^2 \right] \bigg)\bigg). \nonumber
		\end{align}
		In particular, 
		\begin{align}\label{eq:bdbound}
		\dfrac{1}{2p} \left(\mathbb{E}_{Z \sim \nu}[Z] - \mathbb{E}_{Z \sim \mu}[Z] + \mathbb{E}_{Z \sim \mu}[Z^2] - \mathbb{E}_{Z \sim \nu}[Z^2]\right)_+ \leq H(\mu,\nu) \leq \dfrac{2n^2+n}{2p} = O\left(\dfrac{n^2}{p}\right).
		\end{align}
	\end{theorem}
	
	In Theorem \ref{thm:bd} \eqref{eq:bdmain}, we express the optimal transport time $H(\mu,\nu)$ in terms of the moments of $\mu$ and $\nu$. This formulation can perhaps lead to explicit results as we are tackling with moments of $\mu$ and $\nu$ instead of the mean hitting time of birth-death processes. In both \eqref{eq:bdmain} and \eqref{eq:bdbound}, we observe that the optimal transport time $H(\mu,\nu)$ is inversely proportional to the parameter $p$. This is reasonable as $p$ gets smaller (resp.~bigger), the birth-death chain moves slower (resp.~faster) on the state space, making it harder (resp.~easier) to transport $\mu$ to $\nu$. Before we proceed to the proof of Theorem \ref{thm:bd}, we first state a lemma that yields a closed-form expression for the mean hitting time of birth-death processes. As a side remark, we note that closed-form formulae of mean hitting time of birth-death processes has been investigated in \cite{PT96,Tetali91} using the electric network approach.
	
	\begin{lemma}\label{lem:bdhit}
		Suppose that $X$ is a birth-death process with transition dynamics given by \eqref{eq:bd}. Then the mean hitting times can be written as, for $i,j \in \llbracket 0,n \rrbracket$,
		$$\mathbb{E}_i[\tau_j] = \begin{cases}
		\dfrac{(i+j+1)(j-i)}{2p}, \quad i < j,\\
		\dfrac{(i+j-1)(i-j)}{2p}, \quad i > j.
		\end{cases}$$
	\end{lemma}
	
	\begin{proof}
		We first consider the case when $i < j$. Using the strong Markov property and the birth-death property, we see that
		$$\mathbb{E}_i[\tau_j] = \sum_{k=i}^{j-1} \mathbb{E}_k[\tau_{k+1}] = \sum_{k=i}^{j-1} \dfrac{k+1}{p} = \dfrac{(i+j+1)(j-i)}{2p},$$
		where we use the equality $\mathbb{E}_k[\tau_{k+1}] = (k+1)/p$ as in \cite[Section $2.5$]{LPW09}. Next, we treat the case $i > j$. Using again the equality $\mathbb{E}_{k+1}[\tau_{k}] = k/p$ as well as strong Markov property and birth-death property, we have
		$$\mathbb{E}_i[\tau_j] = \sum_{k=j}^{i-1} \mathbb{E}_{k+1}[\tau_{k}] = \sum_{k=j}^{i-1} \dfrac{k}{p} = \dfrac{(i+j-1)(i-j)}{2p}.$$
	\end{proof}
		
	With the above Lemma \ref{lem:bdhit}, we can prove the desired result of Theorem \ref{thm:bd}.
		
	\begin{proof}[Proof of Theorem \ref{thm:bd}]
		Using Theorem \ref{thm:Hmunu}, we have
		\begin{align}\label{eq:bdH}
		H(\mu,\nu) &= \max_{j \in \llbracket 0,n \rrbracket} \sum_{i < j} (\mu_i - \nu_i)\dfrac{(i+j+1)(j-i)}{2p} + \sum_{i > j} (\mu_i - \nu_i)\dfrac{(i+j-1)(i-j)}{2p} \nonumber \\
		&= \dfrac{1}{2p} \max_{j \in \llbracket 0,n \rrbracket} \sum_{i < j} (\mu_i - \nu_i)(j^2 - i^2) + \sum_{i > j} (\mu_i - \nu_i)(i^2 - j^2) + j \sum_{i \neq j} (\mu_i - \nu_i) - \sum_{i \neq j} i(\mu_i - \nu_i).
		\end{align}
		The third and the fourth term can be simplified into
		\begin{align}\label{eq:bdH34}
		j \sum_{i \neq j} (\mu_i - \nu_i) - \sum_{i \neq j} i(\mu_i - \nu_i) = j(\nu_j - \mu_j) + j \mu_j - \mathbb{E}_{Z \sim \mu}[Z] - j\nu_j + \mathbb{E}_{Z \sim \nu}[Z] = \mathbb{E}_{Z \sim \nu}[Z] - \mathbb{E}_{Z \sim \mu}[Z],
		\end{align}
		while the first and the second term reduces to
		\begin{align}\label{eq:bdH12}
		\sum_{i < j} (\mu_i - \nu_i)(j^2 - i^2) + \sum_{i > j} (\mu_i - \nu_i)(i^2 - j^2) &= j^2 \left(\mathbb{E}_{Z \sim \mu}[\1_{\{Z < j\}} - \1_{\{Z > j\}}] - \mathbb{E}_{Z \sim \nu}[\1_{\{Z < j\}} - \1_{\{Z > j\}}]\right) \nonumber \\
		&\quad - \mathbb{E}_{Z \sim \mu}[Z^2\1_{\{Z < j\}} - Z^2\1_{\{Z > j\}}] + \mathbb{E}_{Z \sim \nu}[Z^2\1_{\{Z < j\}} - Z^2\1_{\{Z > j\}}] \nonumber \\
		&= \mathbb{E}_{Z \sim \mu}\left[\max\{Z,j\}^2 - \min\{Z,j\}^2 \right] \nonumber \\
		&\quad - \mathbb{E}_{Z \sim \nu}\left[\max\{Z,j\}^2 - \min\{Z,j\}^2 \right]
		\end{align}
		\eqref{eq:bdmain} follows by plugging \eqref{eq:bdH34} and \eqref{eq:bdH12} back into \eqref{eq:bdH}. For the upper bound of \eqref{eq:bdbound}, using \eqref{eq:bdmain} and the fact that $Z$ is supported on $\llbracket 0,n \rrbracket$, we have
		$$H(\mu,\nu) \leq \dfrac{1}{2p}\bigg(\mathbb{E}_{Z \sim \nu}[Z] + \max_{j \in \llbracket 0,n \rrbracket} \bigg(\mathbb{E}_{Z \sim \mu}\left[\max\{Z,j\}^2\right] + \mathbb{E}_{Z \sim \nu}\left[\min\{Z,j\}^2\right]\bigg)\bigg) \leq \dfrac{2n^2 + n}{2p}.$$
		As for the lower bound of \eqref{eq:bdbound}, we observe that
		$$\max_{j \in \llbracket 0,n \rrbracket} \bigg(\mathbb{E}_{Z \sim \mu}\left[\max\{Z,j\}^2 - \min\{Z,j\}^2 \right] - \mathbb{E}_{Z \sim \nu}\left[\max\{Z,j\}^2 - \min\{Z,j\}^2 \right] \bigg) \geq \mathbb{E}_{Z \sim \mu}\left[ Z^2 \right] - \mathbb{E}_{Z \sim \nu}\left[ Z^2 \right].$$
	\end{proof}
		
\subsection{... winning streak Markov chain}\label{subsec:ws}
		
		In this section, we take the so-called winning streak Markov chain $X$ living on the state space $\mathcal{X} = \llbracket 1,n \rrbracket$. The transition dynamics $P = (p_{i,j})_{i,j \in \mathcal{X}}$ is governed by
		\begin{align}\label{eq:ws}
		p_{i,j} = \begin{cases}
		1/2, \quad \textrm{if } j = i + 1, i \in \llbracket 1,n-1 \rrbracket,\\
		1/2, \quad \textrm{if } j = 1, i \in \llbracket1,n \rrbracket, \\
		1/2, \quad \textrm{if } j = i = n,
		\end{cases}
		\end{align}
		and $0$ otherwise. The expected hitting time of such chain has been calculated in \cite[Example 2]{CZ08}, which is given by
		$$\mathbb{E}_i[\tau_j] = \begin{cases}
		2^j - 2^i, \quad i < j,\\
		2^j, \quad i > j.
		\end{cases}$$
		The winning streak Markov chain is also studied in \cite{LPW09}. In our main result below, we formulate the optimal transport time $H(\mu,\nu)$ in terms of the probability generating function (pgf) of $\mu$ and $\nu$, which subsequently leads to a lower bound \eqref{eq:wsbound} expressed in the differences between the pgf of these two. However, this lower bound may not be tight as the value can be zero.
		
\begin{theorem}\label{thm:ws}
		Suppose that $X$ is a winning streak Markov chain with the above transition dynamics \eqref{eq:ws}. For the optimal transport time $H(\mu,\nu)$ from $\mu$ to $\nu$ using $X$, we have
		\begin{align}\label{eq:wsmain}
		H(\mu,\nu) = \max_{j \in \llbracket 1,n \rrbracket} \mathbb{E}_{Z \sim \nu} \left[2^Z \1_{\{Z \leq j\}} \right] - \mathbb{E}_{Z \sim \mu} \left[2^Z \1_{\{Z \leq j\}} \right].
		\end{align}
		In particular, 
		\begin{align}\label{eq:wsbound}
		\left(\mathbb{E}_{Z \sim \nu} \left[2^Z \right] - \mathbb{E}_{Z \sim \mu} \left[2^Z \right]\right)_+ \leq H(\mu,\nu) \leq 2^n = O\left(2^n\right).
		\end{align}
\end{theorem}
		
\begin{proof}
		Using again Theorem \ref{thm:Hmunu}, we have
		\begin{align*}\label{eq:wsH}
		H(\mu,\nu) &= \max_{j \in \llbracket 1,n \rrbracket} \sum_{i < j} (\mu_i - \nu_i)(2^j - 2^i) + \sum_{i > j} (\mu_i - \nu_i)2^j \nonumber \\
		&= \max_{j \in \llbracket 1,n \rrbracket} 2^j \sum_{i \neq j} (\mu_i - \nu_i) - \sum_{i < j}2^i(\mu_i - \nu_i) \nonumber \\
		&= \max_{j \in \llbracket 1,n \rrbracket} 2^j (\nu_j - \mu_j) + \mathbb{E}_{Z \sim \nu}[2^Z \1_{\{Z < j\}}] - \mathbb{E}_{Z \sim \mu}[2^Z \1_{\{Z < j\}}] \nonumber \\
		&= \max_{j \in \llbracket 1,n \rrbracket} \mathbb{E}_{Z \sim \nu} \left[2^Z \1_{\{Z \leq j\}} \right] - \mathbb{E}_{Z \sim \mu} \left[2^Z \1_{\{Z \leq j\}} \right].
		\end{align*}
		To see the upper bound in \eqref{eq:wsbound}, we simply use Theorem \ref{thm:mcmain} and $\max_{i,j \in \llbracket 1,n \rrbracket} \mathbb{E}_i[\tau_j] = \mathbb{E}_1[\tau_n] = 2^n.$ As for the lower bound in \eqref{eq:wsbound}, we note that
		$$\max_{j \in \llbracket 1,n \rrbracket} \mathbb{E}_{Z \sim \nu} \left[2^Z \1_{\{Z \leq j\}} \right] - \mathbb{E}_{Z \sim \mu} \left[2^Z \1_{\{Z \leq j\}} \right] \geq \mathbb{E}_{Z \sim \nu} \left[2^Z \right] - \mathbb{E}_{Z \sim \mu} \left[2^Z \right].$$
\end{proof}

To see that the upper bound of \eqref{eq:wsbound} is of the correct order in $n$,
we now look at an example in which we take $n > 1$, $\mu = \delta_1$, the Dirac mass at $1$, and $\nu_i = \frac{1}{2(n-1)}$ for $i \in \llbracket 1,n-1 \rrbracket$, $\nu_n = 1/2$. This leads to
$$\max_{j \in \llbracket 1,n \rrbracket} \mathbb{E}_{Z \sim \nu} \left[2^Z \1_{\{Z \leq j\}} \right] =  \mathbb{E}_{Z \sim \nu} \left[2^Z \right] = \dfrac{2^{n-1}-1}{n-1} + 2^{n-1}, \quad \mathbb{E}_{Z \sim \delta_1}[2^Z \1_{\{Z \leq j\}}] = 2.$$ 
As a result, it follows from \eqref{eq:wsmain} that
$$H(\mu,\nu) = \max_{j \in \llbracket 1,n \rrbracket} \mathbb{E}_{Z \sim \nu} \left[2^Z \1_{\{Z \leq j\}} \right] - \mathbb{E}_{Z \sim \mu} \left[2^Z \1_{\{Z \leq j\}} \right] = \dfrac{2^{n-1}-1}{n-1} + 2^{n-1} - 2 = O(2^n).$$
As another remark, if we are allowed to choose a Markov chain $X$ to transport $\mu$ to $\nu$, it seems that the winning streak Markov chain is a poor choice as the worst case optimal transport time can grow exponentially in $n$, the size of the state space.

\subsection{... random walk on $d$-dimensional hypercube}\label{subsec:hc}

In this section, we study a multivariate example, namely a simple random walk on $d$-dimensional hypercube. The state space $\mathcal{X}$ is the set of binary $d$-duples $\{0,1\}^d$, and the size of the state space is $n := |\mathcal{X}| = 2^d$. At each time step, one of the $d$ coordinates, say $x_j$, is chosen uniformly at random and changed to $1-x_j$.
For optimal transport problem on hypercube, our result below shows that the optimal transport time is linear in the size $n$ of the state space (or exponential in the dimension $d$).

\begin{corollary}\label{cor:hypercube}
	Suppose that $X$ is a simple random walk on $d$-dimensional hypercube. For the optimal transport time $H(\mu,\nu)$ from $\mu$ to $\nu$ using $X$, we have
	\begin{align}
	H(\mu,\nu) \leq (1+o(1))n,
	\end{align}
	where $n = 2^d$.
\end{corollary}

\begin{proof}
	According to \cite[Section $5.2$]{P94}, the maximial mean hitting time on $d$-dimensional hypercube is 
	$$\max_{i,j \in \mathcal{X}} \mathbb{E}_i[\tau_j] = \mathbb{E}_{(0,\ldots,0)}[\tau_{(1,\ldots,1)}] = (1+o(1))n.$$
	Desired result follows from Theorem \ref{thm:mcmain}.
\end{proof}

\subsection{... random walks on connected graphs}\label{subsec:graphs}

In this section and subsequent sections, we study random walk on connected graph $G = (V, E)$, where $V$ is the vertex set and $E$ is the edge set. We denote $d_i$ to be the degree of $i \in V$. For random walk on such a graph, the transition dynamics on $\mathcal{X} = V$ is 
\begin{align}\label{eq:rwgraph}
p_{i,j} = \begin{cases}
\dfrac{1}{d_i}, \quad \textrm{if } (i,j) \in E,\\
0, \quad \textrm{otherwise.} 
\end{cases}
\end{align}
In other words, at each vertex $i$ the random walk pick a vertex $j$ that are connected to $i$ with an edge uniformly at random. For an overview of various results of random walk on graph, we refer readers to the excellent survey of \cite{L96}. In our main result of this section, we prove that $H(\mu,\nu)$ is of $O(n^3)$. However, as we will see in subsequent sections, for a variety of special graphs this bound is quite loose.

\begin{corollary}\label{cor:rwgraph}
	Suppose that $X$ is a random walk on connected graph $G = (V,E)$. For the optimal transport time $H(\mu,\nu)$ from $\mu$ to $\nu$ using $X$, we have
	\begin{align*}
	H(\mu,\nu) \leq n(n-1)^2 = O(n^3),
	\end{align*}
	where $n := |V|$.
\end{corollary}

\begin{proof}
	According to \cite[Theorem $2$]{Lawler86}, the maximial mean hitting time on a connected graph is upper bounded by
	$$\max_{i,j \in \mathcal{X}} \mathbb{E}_i[\tau_j] \leq n(n-1)^2.$$
	Desired result follows from Theorem \ref{thm:mcmain}.
\end{proof}

The stationary distribution $\pi = (\pi_i)$ of $X$ is given by, for $i \in V$,
$$\pi_i = \dfrac{d_i}{2|E|}.$$ In our next result of this section, we specialize into the case of random walk on graphs with symmetric hitting times. The cover time of such a random walk has been studied in \cite{DS90, P92}. Notable families of graphs that belong to this class are complete graphs, vertex-transitive graphs and distance-regular graphs. For random walk on graphs with symmetric hitting time,
we can express the optimal transport time in terms of average hitting time as follows.

\begin{proposition}[Random walk with symmetric hitting times]\label{prop:rwsym}
	Suppose that $X$ is a random walk on connected graph $G = (V,E)$ with stationary distribution $\pi$ and symmetric hitting time, i.e. for all $i, j \in V$, $\mathbb{E}_i[\tau_j] = \mathbb{E}_j[\tau_i]$. For the optimal transport time $H(\pi,\nu)$ and $H(\mu,\pi)$, we have
	\begin{align}
	H(\pi,\nu) &= t_{av} - \min_j \sum_i \nu_i \mathbb{E}_i[\tau_j] \leq t_{av}, \label{eq:rwsympinu} \\
	H(\mu,\pi) &= \max_j \sum_i \mu_i \mathbb{E}_i[\tau_j] - t_{av} \leq \max_{i,j} \mathbb{E}_i[\tau_j] - t_{av}, \label{eq:rwsymmupi}
	\end{align}
	where 
	$$n := |V|, \quad t_{av} := \sum_{i,j}\pi_i\pi_j\mathbb{E}_i[\tau_j] = \sum_{i=2}^n \dfrac{1}{1-\lambda_i},$$
	and $1 = \lambda_1 > \lambda_2 \geq \ldots \geq \lambda_n$ are the eigenvalues of the transition matrix $P$ of the random walk.
\end{proposition}

\begin{proof}
	We only prove \eqref{eq:rwsympinu}, as \eqref{eq:rwsymmupi} can be proved similarly. By Theorem \ref{thm:Hmunu}, we see that
	\begin{align*}
		H(\pi,\nu) &= \max_j \sum_i (\pi_i - \nu_i) \mathbb{E}_i[\tau_j] = \max_j \sum_i \pi_i \mathbb{E}_j[\tau_i] - \nu_i \E_i[\tau_j] = t_{av} - \min_j \sum_i \nu_i \mathbb{E}_i[\tau_j] \leq t_{av},
	\end{align*}
	where we use the symmetric hitting time property in the second equality, and the random target lemma (see e.g. \cite[Lemma $10.1$]{LPW09}) in the third equality.
\end{proof}

Comparing \eqref{eq:rwsymmupi} with the general result of Theorem \ref{thm:mcmain}, we see that $t_{av}$ serves as a correction term. In addition, the results in Proposition \ref{prop:rwsym} demonstrate a connection with the eigenvalues of the underlying random walk. 

\subsection{... random walk on the $n$-path}\label{subsec:npath}

In this section, we specialize into the case of random walk on the $n$-path. This is one among many classical examples of random walk on graphs that have been studied intensively, see for example \cite{AF14} or \cite{L96}. It can be regarded as a birth-death process on $\llbracket 0,n \rrbracket$ with reflecting boundaries at $0$ and $n$, where its transition dynamic is given by
\begin{align}\label{eq:npath}
p_{i,j} = \begin{cases}
1/2, \quad \textrm{if } j = i \pm 1, i \in \llbracket 1,n-1 \rrbracket,\\
1, \quad \textrm{if } i = 0, j = 1 \, \textrm{or } i = n, j = n-1, \\
0, \quad \textrm{otherwise.} 
\end{cases}
\end{align}

\begin{theorem}\label{thm:npath}
	Suppose that $X$ is a random walk on the $n$-path with transition given by \eqref{eq:npath}. For the optimal transport time $H(\mu,\nu)$ from $\mu$ to $\nu$ using $X$, we have
	\begin{align}\label{eq:npathmain}
	H(\mu,\nu) = \mathbb{E}_{Z \sim \nu}[Z^2] - \mathbb{E}_{Z \sim \mu}[Z^2] + 2n \max_{j \in \llbracket 0,n \rrbracket} \left( \mathbb{E}_{Z \sim \mu} (Z-j)_+ - \mathbb{E}_{Z \sim \nu} (Z-j)_+ \right).
	\end{align}
	In particular, 
	\begin{align}\label{eq:npathbound}
	\left(\mathbb{E}_{Z \sim \nu}[Z^2] - \mathbb{E}_{Z \sim \mu}[Z^2] + 2n \left( \mathbb{E}_{Z \sim \mu}[Z] - \mathbb{E}_{Z \sim \nu}[Z] \right)\right)_+ \leq H(\mu,\nu) \leq n^2 = O\left(n^2\right).
	\end{align}
\end{theorem}

As our first remark, we note that this upper bound \eqref{eq:npathbound} of the $n$-path is tighter than the general $O(n^3)$ bound of Corollary \ref{cor:rwgraph}. When $\mu = \delta_0$ and $\nu = \delta_n$, the upper bound in \eqref{eq:npathbound} is trivally attained as $H(\delta_0, \delta_n) = H(0,n) = \mathbb{E}_0[\tau_n] = n^2$. To demonstrate a non-trivial instance in which both the upper and lower bounds in \eqref{eq:npathbound} are of the same order in $n$, i.e. $H(\mu,\nu) = \Theta(n^2)$, we consider the case when $\mu$ is a discrete uniform on $\llbracket 0,n \rrbracket$ while $\nu$ is binomially distributed with parameters $n$ and $p \in (0,1-1/\sqrt{3})$. In this case, we have 
\begin{align*}
\mathbb{E}_{Z \sim \mu}[Z] &= \dfrac{n}{2}, \quad \mathbb{E}_{Z \sim \mu}[Z^2] = \dfrac{n(2n+1)}{6},\\
\mathbb{E}_{Z \sim \nu}[Z] &= np, \quad \mathbb{E}_{Z \sim \nu}[Z^2] = np(1-p)+n^2p^2.
\end{align*}
As a result, the lower bound in \eqref{eq:npathbound} now reads
$$\mathbb{E}_{Z \sim \nu}[Z^2] - \mathbb{E}_{Z \sim \mu}[Z^2] + 2n \left( \mathbb{E}_{Z \sim \mu}[Z] - \mathbb{E}_{Z \sim \nu}[Z] \right) = n^2\left(p^2 - 2p + \dfrac{2}{3}\right) + n \left(p(1-p) - \dfrac{1}{6}\right).$$
For $p \in (0,1-1/\sqrt{3})$, $p^2 - 2p + 2/3 > 0$ and so both the lower and upper bound of \eqref{eq:npathbound} gives an order of $n^2$. Before we state the proof of Theorem \ref{thm:npath}, we first give a lemma that writes out the explicit expression for the mean hitting time of the random walk on the $n$-path.

\begin{lemma}\label{lem:npathhit}
	Suppose that $X$ is a random walk on the $n$-path with transition given by \eqref{eq:npath}. Then the mean hitting times can be written as, for $i,j \in \llbracket 0,n \rrbracket$,
	$$\mathbb{E}_i[\tau_j] = \begin{cases}
	j^2 - i^2, \quad i < j,\\
	(i-j)2n - (i^2-j^2), \quad i > j.
	\end{cases}$$
\end{lemma}

\begin{proof}
	For $i < j$, the result can be found in \cite[Chapter $5$ Example $8$]{AF14}. For the case of $i > j$, we adapt a similar proof and proceed by induction on $i$. When $i = j+1$, we observe that $\mathbb{E}_{j+1}[\tau_j]$ is one less than the expected return time from $j$ to $j$ in a graph of $n-j+1$ nodes starting on the left, which is $2(n-j)-1$. Suppose that $\mathbb{E}_i[\tau_j] = (i-j)2n - (i^2-j^2)$ holds for some $i$, then
	$$\mathbb{E}_{i+1}[\tau_j] = \mathbb{E}_{i+1}[\tau_i] + \mathbb{E}_i[\tau_j] = 2(n-i) - 1 + (i-j)2n - (i^2-j^2) = (i+1-j)2n - ((i+1)^2-j^2).$$
\end{proof}

\begin{proof}[Proof of Theorem \ref{thm:npath}]
	Using Theorem \ref{thm:Hmunu} together with Lemma \ref{lem:npathhit}, we have
	\begin{align}\label{eq:npathH}
	H(\mu,\nu) &= \max_{j \in \llbracket 0,n \rrbracket} \sum_{i < j} (\mu_i - \nu_i)(j^2 - i^2) + \sum_{i > j} \left((\mu_i - \nu_i)(i-j)2n + (j^2-i^2) \right) \nonumber \\
	&= \max_{j \in \llbracket 0,n \rrbracket} \sum_{i} (\mu_i - \nu_i)(j^2 - i^2) + 2n \left(\mathbb{E}_{Z \sim \mu}(Z-j)_+ - \mathbb{E}_{Z \sim \nu}(Z-j)_+ \right) \nonumber \\
	&= \mathbb{E}_{Z \sim \nu}[Z^2] - \mathbb{E}_{Z \sim \mu}[Z^2] + 2n \max_{j \in \llbracket 0,n \rrbracket} \left( \mathbb{E}_{Z \sim \mu} (Z-j)_+ - \mathbb{E}_{Z \sim \nu} (Z-j)_+ \right). \nonumber
	\end{align}
	For the upper bound of \eqref{eq:npathbound}, we apply Theorem \ref{thm:mcmain} and Lemma \ref{lem:npathhit} to see that
	$$\max_{i,j \in \llbracket 0,n \rrbracket} \mathbb{E}_i[\tau_j] = \mathbb{E}_{0}[\tau_{n}] = n^2.$$
	As for the lower bound of \eqref{eq:npathbound}, we use \eqref{eq:npathmain} and note that
	$$2n \max_{j \in \llbracket 0,n \rrbracket} \left( \mathbb{E}_{Z \sim \mu} (Z-j)_+ - \mathbb{E}_{Z \sim \nu} (Z-j)_+ \right) \geq 2n \left( \mathbb{E}_{Z \sim \mu}[Z] - \mathbb{E}_{Z \sim \nu}[Z] \right).$$
\end{proof}

\subsection{... random walk on complete graphs}\label{subsec:cgraph}

In this section, we look into the example of random walk on complete graphs on $\mathcal{X} = \llbracket 0,n \rrbracket$, with transition dynamics driven by, for $i \neq j \in \mathcal{X}$, $p_{i,i} = 0$ and $p_{i,j} = 1/n$, and the stationary distribution is well-known to be a discrete uniform $\pi_i = 1/(n+1)$. Alternatively, we can take the degree of node $i$ as $d_i = n$ in \eqref{eq:rwgraph}. According to \cite[Chapter $5$ Example $9$]{AF14}, the mean hitting times of this random walk are
$$\E_i[\tau_j] = n.$$
As our main result below illustrates, the optimal transport time is of the order $O(n)$.
\begin{theorem}\label{thm:cgraph}
	Suppose that $X$ is a random walk on the complete graph. For the optimal transport time $H(\mu,\nu)$ from $\mu$ to $\nu$ using $X$, we have
	\begin{align}
	H(\mu,\nu) = n \max_{j \in \llbracket 0,n \rrbracket} (\nu_j - \mu_j).
	\end{align}
	In particular, for $j \in \llbracket 0,n \rrbracket$,
	\begin{align}\label{eq:cgraphbound}
	n (\nu_j - \mu_j)_+ \leq H(\mu,\nu) \leq n ||\mu - \nu||_{TV} \leq n = O(n).
	\end{align}
	The distribution within the class of Dirac mass $\mathcal{P} := \{\delta_i;~ i \in \mathcal{X}\}$ that minimizes $H(\mu,\nu)$ is $\mu^* := \delta_{i^*}$ where $i^* := \arg \max \nu_i$, i.e.
	\begin{align}\label{eq:fastcgraph}
	\min_{\mu \in \mathcal{P}} H(\mu,\nu) = H(\mu^*,\nu).
	\end{align}
\end{theorem}

We first note that the $O(n)$ bound of $H(\mu,\nu)$ is much tighter than the $O(n^3)$ bound in the general result Corollary \ref{cor:rwgraph}. In addition, \eqref{eq:fastcgraph} has an operational meaning when we are given the freedom to choose within $\mathcal{P}$ such that its optimal transport time $H(\mu,\nu)$ is the fastest: we should pick the state $i$ which maximizes $\nu_i$. This also makes sense intuitively as we initialize the chain at a high probability state of $\nu$, thus the transport time should be faster than other initializations.

\begin{proof}[Proof of Theorem \ref{thm:cgraph}]
	Thanks to Theorem \ref{thm:Hmunu}, we have
	\begin{align*}
	H(\mu,\nu) &= \max_{j \in \llbracket 0,n \rrbracket} \sum_{i \neq j} (\mu_i - \nu_i) \E_i[\tau_j] = n \max_{j \in \llbracket 0,n \rrbracket} (\nu_j - \mu_j).
	\end{align*}
	For the upper bound in \eqref{eq:cgraphbound}, we note that 
	$$H(\mu,\nu) \leq n \max_{j \in \llbracket 0,n \rrbracket} |\nu_j - \mu_j| \leq n ||\mu - \nu||_{TV} \leq n.$$
	As for the lower bound in \eqref{eq:cgraphbound}, we simple use
	$$n \max_{j \in \llbracket 0,n \rrbracket} (\nu_j - \mu_j) \geq n (\nu_j - \mu_j).$$
	Finally, we prove \eqref{eq:fastcgraph}. For $\mu \in \mathcal{P}$ and $\mu \neq \mu^*$, we see that
	$$H(\mu,\nu) = n \max_{j \in \llbracket 0,n \rrbracket} \nu_j \geq n \max_{j \in \llbracket 0,n \rrbracket, j \neq i^*} \nu_j = H(\mu^*,\nu).$$
\end{proof}

\subsection{... random walk on the $n$-star}\label{subsec:nstar}

In this section, we consider the $n$-star graph on $\mathcal{X} = \llbracket 0,n \rrbracket$, where there is an edge connecting $0 - 1, 0 - 2, \ldots, 0 - n$ and $0$ is the center of the star.

\begin{theorem}\label{thm:nstar}
	Suppose that $X$ is a random walk on the $n$-star. For the optimal transport time $H(\mu,\nu)$ from $\mu$ to $\nu$ using $X$, we have
	\begin{align}
	H(\mu,\nu) = \nu_0 - \mu_0 + 2n \left(\max_{j \in \llbracket 1,n \rrbracket} (\nu_j - \mu_j)\right)_+.
	\end{align}
	In particular, for $j \in \llbracket 1,n \rrbracket$,
	\begin{align}\label{eq:nstarbound}
	\nu_0 - \mu_0 + 2n (\nu_j - \mu_j)_+ \leq H(\mu,\nu) \leq (2n+1) ||\mu - \nu||_{TV} \leq 2n + 1 = O(n).
	\end{align}
\end{theorem}

\begin{proof}
	As in previous sections, the mean hitting times are vital for us in determining $H(\mu,\nu)$. According to \cite[Chapter $5$ Example $10$]{AF14}, the mean hitting times of the random walk on $n$-star are, for $i \in \llbracket 1,n \rrbracket$,
	$$\mathbb{E}_i[\tau_j] = \begin{cases}
	1, \quad j = 0,\\
	2n, \quad j \neq i,0.
	\end{cases}$$
	As for $i = 0$ and $j \in \llbracket 1,n \rrbracket$, we have
	$$\mathbb{E}_0[\tau_j] = 2n-1.$$
	This leads to, for $j \in \llbracket 1,n \rrbracket$,
	\begin{align*}
	\sum_i (\mu_i - \nu_i) \E_i[\tau_0] &= \nu_0 - \mu_0, \\
	\sum_i (\mu_i - \nu_i) \E_i[\tau_j] &= (2n-1)(\mu_0 - \nu_0) + 2n(\nu_0  + \nu_j - \mu_0 - \mu_j) = \nu_0 - \mu_0 + 2n(\nu_j - \mu_j). 
	\end{align*}
	As a result, according to Theorem \ref{thm:Hmunu} we have
	$$H(\mu,\nu) = \nu_0 - \mu_0 + 2n \left(\max_{j \in \llbracket 1,n \rrbracket} (\nu_j - \mu_j)\right)_+,$$
	and \eqref{eq:nstarbound} follows from 
	$$(\nu_j - \mu_j)_+ \leq \max_{j \in \llbracket 0,n \rrbracket} (\nu_j - \mu_j) \leq ||\mu-\nu||_{TV}.$$
\end{proof}

\bibliographystyle{abbrvnat}
\bibliography{thesis}

\begin{thebibliography}{14}
\providecommand{\natexlab}[1]{#1}
\providecommand{\url}[1]{\texttt{#1}}
\expandafter\ifx\csname urlstyle\endcsname\relax
  \providecommand{\doi}[1]{doi: #1}\else
  \providecommand{\doi}{doi: \begingroup \urlstyle{rm}\Url}\fi

\bibitem[Aldous and Fill(2002)]{AF14}
D.~Aldous and J.~A. Fill.
\newblock Reversible {M}arkov {C}hains and {R}andom {W}alks on {G}raphs, 2002.
\newblock Unfinished monograph, recompiled 2014, available at
  \url{http://www.stat.berkeley.edu/~aldous/RWG/book.html}.

\bibitem[Chen and Zhang(2008)]{CZ08}
H.~Chen and F.~Zhang.
\newblock The expected hitting times for finite {M}arkov chains.
\newblock \emph{Linear Algebra Appl.}, 428\penalty0 (11-12):\penalty0
  2730--2749, 2008.

\bibitem[Devroye and Sbihi(1990)]{DS90}
L.~Devroye and A.~Sbihi.
\newblock Random walks on highly symmetric graphs.
\newblock \emph{J. Theoret. Probab.}, 3\penalty0 (4):\penalty0 497--514, 1990.

\bibitem[Lawler(1986)]{Lawler86}
G.~F. Lawler.
\newblock Expected hitting times for a random walk on a connected graph.
\newblock \emph{Discrete Math.}, 61\penalty0 (1):\penalty0 85--92, 1986.

\bibitem[Levin et~al.(2009)Levin, Peres, and Wilmer]{LPW09}
D.~A. Levin, Y.~Peres, and E.~L. Wilmer.
\newblock \emph{Markov chains and mixing times}.
\newblock American Mathematical Society, Providence, RI, 2009.

\bibitem[Lov\'asz(1996)]{L96}
L.~Lov\'asz.
\newblock Random walks on graphs: a survey.
\newblock In \emph{Combinatorics, {P}aul {E}rd\H os is eighty, {V}ol.\ 2
  ({K}eszthely, 1993)}, volume~2 of \emph{Bolyai Soc. Math. Stud.}, pages
  353--397. J\'anos Bolyai Math. Soc., Budapest, 1996.

\bibitem[Lov\'{a}sz and Winkler(1995)]{LW95}
L.~Lov\'{a}sz and P.~Winkler.
\newblock Efficient stopping rules for markov chains.
\newblock In \emph{Proceedings of the Twenty-seventh Annual ACM Symposium on
  Theory of Computing}, STOC '95, pages 76--82, New York, NY, USA, 1995. ACM.

\bibitem[Lov\'asz and Winkler(1998)]{LW98}
L.~Lov\'asz and P.~Winkler.
\newblock Reversal of {M}arkov chains and the forget time.
\newblock \emph{Combin. Probab. Comput.}, 7\penalty0 (2):\penalty0 189--204,
  1998.

\bibitem[M\"orters and Redl(2016)]{MR16}
P.~M\"orters and I.~Redl.
\newblock Skorokhod embeddings for two-sided {M}arkov chains.
\newblock \emph{Probab. Theory Related Fields}, 165\penalty0 (1-2):\penalty0
  483--508, 2016.

\bibitem[Ob\l\'oj(2004)]{O04}
J.~Ob\l\'oj.
\newblock The {S}korokhod embedding problem and its offspring.
\newblock \emph{Probab. Surv.}, 1:\penalty0 321--390, 2004.

\bibitem[Palacios(1992)]{P92}
J.~L. Palacios.
\newblock Expected cover times of random walks on symmetric graphs.
\newblock \emph{J. Theoret. Probab.}, 5\penalty0 (3):\penalty0 597--600, 1992.

\bibitem[Palacios(1994)]{P94}
J.~L. Palacios.
\newblock Expected hitting and cover times of random walks on some special
  graphs.
\newblock In \emph{Proceedings of the {F}ifth {I}nternational {S}eminar on
  {R}andom {G}raphs and {P}robabilistic {M}ethods in {C}ombinatorics and
  {C}omputer {S}cience ({P}ozna\'n, 1991)}, volume~5, pages 173--182, 1994.

\bibitem[Palacios and Tetali(1996)]{PT96}
J.~L. Palacios and P.~Tetali.
\newblock A note on expected hitting times for birth and death chains.
\newblock \emph{Statist. Probab. Lett.}, 30\penalty0 (2):\penalty0 119--125,
  1996.

\bibitem[Tetali(1991)]{Tetali91}
P.~Tetali.
\newblock Random walks and the effective resistance of networks.
\newblock \emph{J. Theoret. Probab.}, 4\penalty0 (1):\penalty0 101--109, 1991.

\end{thebibliography}

\end{document}